\documentclass[leqno,11pt]{amsart}
\usepackage{amssymb, amsmath,amsmath,latexsym,amssymb,amsfonts,amsbsy, amsthm,esint }
\usepackage{color}
\usepackage{graphicx}
\usepackage{hyperref}

\makeatletter
\DeclareFontFamily{U}{tipa}{}
\DeclareFontShape{U}{tipa}{m}{n}{<->tipa10}{}
\newcommand{\arc@char}{{\usefont{U}{tipa}{m}{n}\symbol{62}}}%

\newcommand{\arc}[1]{\mathpalette\arc@arc{#1}}

\newenvironment{problem}[2][Problem]{\begin{trivlist}
\item[\hskip \labelsep {\bfseries #1}\hskip \labelsep {\bfseries #2.}]}{\end{trivlist}}

\newcommand{\arc@arc}[2]{%
  \sbox0{$\m@th#1#2$}%
  \vbox{
    \hbox{\resizebox{\wd0}{\height}{\arc@char}}
    \nointerlineskip
    \box0
  }%
}
\makeatother

\let\pa=\partial
\let\al=\alpha

\let\d=\delta

\let\lam=\lambda

\let\f=\frac

\let\D=\Delta

\let\Om=\Omega
\let\e=\varepsilon
\let\pa=\partial

\let\ri=\rightarrow

\let\na=\nabla

\def\BR{\mathbb{R}}

\newcommand{\beq}{\begin{equation}}
\newcommand{\eeq}{\end{equation}}
\newcommand{\beqo}{\begin{equation*}}
\newcommand{\eeqo}{\end{equation*}}
\newcommand{\ben}{\begin{eqnarray}}
\newcommand{\een}{\end{eqnarray}}
\newcommand{\beno}{\begin{eqnarray*}}
\newcommand{\eeno}{\end{eqnarray*}}

\newtheorem{theorem}{Theorem}[section]

\newtheorem{lemma}[theorem]{Lemma}

\newtheorem{Theorem}{Theorem}[section]

\newtheorem{Remark}[Theorem]{Remark}

\theoremstyle{remark}
\newtheorem{step}{Step}

\begin{document}

\title[$l^1$ minimal partitions for Dirichlet eigenvalue]{Large m
asymptotics for minimal partitions of the Dirichlet eigenvalue}

\author{Zhiyuan Geng}
\address{Courant Institute of Mathematical Sciences, New York University}
\email{zhiyuan@cims.nyu.edu}

\author{Fanghua Lin}
\address{Courant Institute of Mathematical Sciences, New York University}
\email{linf@cims.nyu.edu}

\begin{abstract}
In this paper, we study large $m$ asymptotics of the $l^1$ minimal $m$-partition problem for Dirichlet eigenvalue. For any smooth domain $\Omega\subset \mathbb{R}^n$ such that $|\Omega|=1$, we prove that the limit $\lim\limits_{m\rightarrow\infty}l_m^1(\Omega)=c_0$ exists, and the constant $c_0$ is independent of the shape of $\Omega$. Here $l_m^1(\Omega)$ denotes the minimal value of the normalized sum of the first Laplacian eigenvalues for any $m$-partition of $\Omega$.

\end{abstract}

\maketitle


\section{Introduction}

Let $\Om$ be a bounded, smooth domain in $\mathbb{R}^n$, and $m>1$ be a positive integer. We consider the following so-called $l^1$-minimal partition problem:
\begin{problem}{P}
Find a partition of $\Omega$ into $m$, mutually disjoint subsets $\Om_j$, $j=1,2,\ldots,m$, such that $\Omega=\bigcup_{j=1}^m\Om_j$, and it minimizes the
$l^1$ energy functional $\sum_{j=1}^m \lam_1(\Om_j)$ among all admissible partitions. Here $\lam_1(A)$ denotes the first eigenvalue of Laplacian $\D$
on $A$ with the zero Dirichlet boundary condition on $\pa A$.
\end{problem}
The existence of the minimal partition and regularity of free interfaces have been studied by many authors, see \cite{BZ,BD,Sv,BBH,B,MP,CL1,CL2} and
survey articles \cite{He,BH,Lin}. In \cite{CL1}, Cafferelli and the second author proved the equivalence between Problem P and the following problem:

\begin{problem}{P*} Let
\beqo
\Sigma^m=\{y\in\BR^m, \sum\limits_{k\neq l} y_k^2y_l^2=0\}
\eeqo
Find $u\in H^1_0(\Om,\Sigma^m)$ such that
\beqo
\int_\Om u_j^2\,dx=1\text{ for any }j=1,...,m,
\eeqo
and that $u$ minimize $\int_\Om |\na u|^2\,dx$ among all such maps in $H_0^1(\Om,\Sigma^m)$.
\end{problem}

Problem (P*) obviously admits a minimizer $u=(u_1,u_2,...,u_m)$. It is proved in \cite{CL1} that $u$ is locally Lipschitz continuous in $\Om$ (and
Lipschitz continuous upto the boundary when $\pa\Om$ is smooth), and
$\Om_j=\{x\in\Om: u_j(x)>0\}\,(j=1,...,m)$ are open subsets of $\Om$ whose boundaries $\pa\Om_j$ are smooth away from a relatively closed subset
$S\subset\Om$, of Hausdorff dimension at most $n-2$. Moreover $\{\Om_j\}_{j=1}^m$ gives a partition of $\Om$ that minimizes $\sum_{j=1}^m\lam_1(\Om_j)$. It is shown later
by O. Alper that the set $S$ is rectifiable and of bounded $(n-2)$ dimensional Hausdorff measure, \cite{Alper}.

In this note we are interested in the asymptotic behavior of the minimal partition as $m\ri \infty$. Our main theorem is the following:
\begin{theorem}\label{main}
Let $\Om$ be a bounded, smooth domain in $\BR^n$ with $|\Om|=1$. Then
\beq\label{limit}
\lim\limits_{m\ri\infty}l_m^1(\Om)=c_0\quad \text{for some positive constant }c_0 \text{ independent of }\Om.
\eeq
Here
\begin{align*}
l_m^1(\Om)&=\frac{\sum_{j=1}^m\lam_1(\Om_j)}{m^{1+\frac{2}{n}}},\\
\Om&=\bigcup\limits_{j=1}^m \Om_j \text{ is a }l^1\text{-minimal }m\text{-partition}.
\end{align*}
\end{theorem}
\begin{Remark}
For $\Om\subset \BR^2$, by hexagonal tiling construction and Faber-Krahn inequality, one can easily get the following lower bound and upper bound for the constant $c_0$:
\beq\label{boundc0}
\lam_1(D)\leq c_0\leq \lam_1(H),
\eeq
where $D$ is the 2-D unit-area disk and $H$ is the unit-area regular hexagon.
\end{Remark}
It should be noted, in the above theorem, the smoothness of $\Om$ does not play any role here, and the smoothness assumption is just for convenience. The problem of large $m$ asymptotics was considered first in \cite{CL1} and they prove that $\sum_{j=1}^m \lam_1(\Om_j)\simeq m\lam_m(\Om)$, where $\lam_m(\Om)$ is the $m$-th Dirichlet eigenvalue of $\Om$. They also made a conjecture that the limit $\lim\limits_{m\ri\infty}l_m^1(\Om)$ exists and for the case $\Om\subset \BR^2$, the minimal partitions for large $m$ will be close to a regular Hexagon packing pattern and the constant $c_0$ equals to $\lam_1(H)$. Theorem \ref{main} here verifies the first part of the conjecture, while the second part (regular Hexagon pattern) remains open though one can very well expect it in a stochastic sense. In recent years some attempts have been made to a related issue. For examples, Bourgain \cite{Bou} and Steinerberger \cite{St} have improved the lower bound in \eqref{boundc0} by showing that $l_m^1(\Om)>\lam_1(D)+\e_0$ for some sufficiently small constant $\e_0$. Their tools are a quantitative Faber-Krahn inequality and some packing properties of disks in $\BR^2$. In \cite{BFVV}, Bucur, Fragal\`{a}, Velichkov and Verzini study this so-called ``honeycomb conjecture", and they give a proof under the assumption that every $\Om_j\,(j=1,...,m)$  is convex and regular hexagon minimizes $\lam_1$ among all convex hexagons with the same area, which is itself an interesting open problem.

In Section 2 we will prove Theorem \ref{main}. The proof will be concentrated on the case $n=2$. For $n\geq 3$ one can apply the same arguments with only some obvious modifications. We first prove the limit exists for the unit cube, and then we prove the statement for general domain $\Om$ by approximating it using smaller dyadic cubes of the same size.

The research of authors is partially supported by the NSF-Grants, DMS-1501000 and DMS-1955249.

\section{Proof of Theorem \ref{main}}
We first prove the following lemma:
\begin{lemma}\label{cube}
Let $Q$ be a unit cube in $\BR^2$. For any $m>0$ and $k\geq 1$, it holds that
\beq\label{cubeineq}
l_m^1(Q)\geq l_{mk^2}^1(Q).
\eeq
\end{lemma}
\begin{proof}
Let $s=l_m^1(Q)$. By existence of the $l^1$-minimal $m$-partition, there is a $m$-partition $\{\Om_j\}_{j=1}^m$ of $Q$ such that $\sum_{j=1}^m \lam_1(\Om_j)=m^2s$. Now we divided $Q$ into $k^2$ identical cubes $\{Q^i\}_{i=1}^{k^2}$ with edge length $\f{1}{k}$. In each $Q^i$,
we put a translated and scaled copy of the same $m$-partition as $\{\Om_j\}_{j=1}^m$, which is denoted by $\{\Om^i_j\}_{j=1}^m$. As a result we get
a $mk^2$-partition of $Q$, and we have
\beqo
l_{mk^2}^1(Q)\leq \bigg(\sum\limits_{1\leq i\leq k^2}\big(\sum\limits_{1\leq j\leq m} \lam_1(\Om^i_j)\big)\bigg)/(mk^2)^2= s.
\eeqo
Here we have used the degree $-2$-homogeneity of $\lam_1$ with respect to scalings.
\end{proof}

\begin{proof}[Proof of Theorem \ref{main}]
\setcounter{step}{0}
\begin{step}
Consider the unit cube $Q$, we show that there exists $c_0$ such that
\beqo
\lim\limits_{m\ri\infty} l_m^1(Q)=c_0.
\eeqo
Define
\beqo
a(Q)=\liminf\limits_{m\ri\infty} l_m^1(Q).
\eeqo
For any $\e>0$, there exists an integer $m_\e$ such that $l_{m_\e}^1(Q)\leq a(Q)+\f{\e}{2}$. For any $m\geq m_\e$, there exists $k\in \mathbb{N}$ such that $k^2m_\e\leq m\leq (k+1)^2m_\e$. By Lemma \ref{cube}, we have
\beqo
l_{(k+1)^2m_\e}^1(Q)\leq l_{m_\e}^1(Q).
\eeqo
Let $\{\Om_j\}_{j=1}^{(k+1)^2m_\e}$ be the minimal $(k+1)^2m_\e$-partition of $Q$. By grouping together some of the subdomains $\Om_j$, we can obtain a
new  $m$-partition of $Q$, denoted by $\{\Om'_j\}_{j=1}^m$, then we deduce that
\beqo
l_m^1(Q)\leq \frac{\sum\limits_{j=1}^m\lam_1(\Om'_j)}{m^2}\leq \f{((k+1)^2m_\e)^2}{m^2}l_{m_\e}^1(Q)\leq (\f{k+1}{k})^4(a(Q)+\f{\e}{2}).
\eeqo
Let $k_\e$ be sufficiently large such that $(\f{k_\e+1}{k_\e})^4(a(Q)+\f{\e}{2})\leq a(Q)+\e$. Then for any $m\geq k_\e^2m_\e$, we have
\beqo
l_m^1(Q)\leq a(Q)+\e,
\eeqo
which implies that $\limsup\limits_{m\ri\infty} l_m^1(Q)\leq a(Q)$.\\
One can deduce from the above proof that $\lim\limits_{m\ri\infty}l_m^1(Q) = \lim\limits_{m\ri\infty}l_{(m + o(m))}^1(Q)$.

\end{step}
\begin{step}
For any bounded, smooth domain $\Om\subset \BR^2$ such that $|\Om|=1$, we prove
\beqo
\limsup\limits_{m\ri\infty} l_m^1(\Om)\leq \lim\limits_{m\ri\infty}l_m^1(Q)=a(Q).
\eeqo
For any $\e>0$, there is $k\in \mathbb{N}$, such that
\beq\label{approximationbycubes}
\cup_{j=1}^k Q_j\subset \Om\subset \left( \cup_{j=1}^k Q_j \right)\cup \left( \cup_{i=1}^l Q_{k+i} \right).
\eeq
Here $\{Q_j\}_{j=1}^k$ and $\{Q_{k+i}\}_{i=1}^l$ are smaller dyadic cubes of the same size and satisfies
\beqo
\sum\limits_{i=1}^l |Q_{k+i}|\leq \f{\e}{4}.
\eeqo
If $ k > 1$, we
let $m=(k-1)n+t$, where $m,n,t\in\mathbb{N}$ and $t<(k-1)$. Then we have
\begin{align*}
l_m^1(\Om) &\leq l_m^1(\cup_{j=1}^k Q_j)\\
&\leq \left(\frac{(k-1)n^2l_n^1(Q)}{|Q_j|}+\frac{t^2l_t^1(Q)}{|Q_j|}\right)/m^2
\end{align*}
Here the second inequality comes from the construction of the partition that divide each of $Q_j\, (j=1,...,(k-1))$ into $n$ subdomains and divide
the last cube $Q_k$ into $t$ subdomains. Let $n$ be sufficiently large or equivalently $m$ sufficiently large, we can guarantee that the value of the
last line is less than $a(Q)(1+\e)$, which leads to that
$\limsup\limits_{m\ri\infty} l_m^1(\Om)\leq a(Q)$.
\end{step}
\begin{step}
We are left to prove $\liminf\limits_{m\ri\infty} l_m^1(\Om)\geq a(Q)$. Given $\e>0$, by \eqref{approximationbycubes}, $\Om$ can be approximated by smaller dyadic cubes. Then we have
\beqo
l_m^1(\Om)\geq l_m^1\big((\cup_{j=1}^k Q_j) \cup  (\cup_{i=1}^l Q_{k+i}) \big)
\eeqo
It suffices to show that given $m$ large enough,
\beq\label{liminfbound}
l_m^1\big((\cup_{j=1}^k Q_j) \cup  (\cup_{i=1}^l Q_{k+i})\big)\geq (1-\e)a(Q).
\eeq

Actually, \eqref{liminfbound} is implied by the following Lemma \ref{twodomain}.

\begin{lemma}\label{twodomain}
Let $\Om$ be a domain in $\BR^2$ with $|\Om|=1$. $\Gamma$ is a straight line that separates $\Om$ into two sub-domains $D_1,\,D_2$, with area $\al,1-\al$ respectively. Assume there exists a constant $c$ such that
\beqo
\lim\limits_{m\ri\infty} l_m^1(\f{1}{\sqrt{\al}}D_1)=\lim\limits_{m\ri\infty} l_m^1(\f{1}{\sqrt{1-\al}}D_2)=c
\eeqo
Then
\beqo
\lim\limits_{m\ri\infty} l_m^1(\Om)=c.
\eeqo
\end{lemma}
Let's assume this lemma and proceed with our proof. Note that $(\cup_{j=1}^k Q_j) \cup  (\cup_{i=1}^l Q_{k+i})$ is the union of $k+l$ small cubes,
whose areas added up to $(1+\d)$ for
some $\d<\frac{\e}{4}$. By proper scalings and by repetitive applications of Lemma \ref{twodomain}, we can then get that $\lim\limits_{m\ri\infty}
l_m^1\big((\cup_{j=1}^k Q_j) \cup  (\cup_{i=1}^l Q_{k+i})\big)=\frac{1}{1+\d}\lim\limits_{m\ri\infty} l_m^1(Q)\geq (1-\e)a(Q)$, which yields the
conclusion \eqref{liminfbound}. The proof of the theorem is then completed.
\end{step}
\end{proof}

\begin{proof}[Proof of Lemma \ref{twodomain}]
Without loss of generality, one assumes $\Gamma=\{x=0\}$ and $D_1=\{z=(x,y)\in \Om: x<0\}$, $D_2=\{z=(x,y)\in \Om: x>0\}$. Note that by the same
arguments as in the Step 2, we have
\beqo
\limsup\limits_{m\ri\infty} l_m^1(\Om)\leq c.
\eeqo
It suffices to prove for any $\e>0$, there exists $m_\e$ such that if $m\geq m_\e$, then
\beq\label{finalineq}
l_m^1(\Om)\geq c(1-\e).
\eeq

In the rest of proof we will always fix $\e>0$ and we always assume $m$ is large enough (depending on $\e$ that it will be specified later). We need to
study
Problem (P*), which is the equivalent formulation of the minimal partition Problem P. Let $u=(u_1,...,u_m)\in H_0^1(\Om,\Sigma^m)$ be a minimizer
of Problem (P*), then $\{\mathrm{supp}(u_j)\}_{j=1}^m$ gives a minimal $m$-partition of $\Om$. Denote
\beqo
\Om_j=\mathrm{supp}(u_j)
\eeqo
Take a fixed small number $\d$ (also depends on $\e$ only, will be determined later). We define the following regions:
\begin{align*}
S_\d=\{z=(x,y)&\in\Om: \mathrm{dist}(z,\Gamma)< \f{\d}{2}\}=\{z=(x,y)\in \Om, |x|<\f{\d}{2}\}\\
&D_1'=D_1\backslash S_\d,\quad D_2'=D_2\backslash S_\d
\end{align*}
Then we classify the subdomains in the partition $\{\Om_j\}_{j=1}^m$ according to their intersections with $S_\d,\, D_1',\,D_2'$.
\begin{align*}
A_\d&=\{\Om_k:\, \Om_k\cap D_2'=\emptyset\},\\
B_\d&=\{\Om_k:\, \Om_k\cap D_1'=\emptyset \},\\
C_\d&=\{\Om_k: \, \Om_k\cap D_1'\neq \emptyset,\; \Om_k\cap D_2'\neq \emptyset\}.
\end{align*}

We are mostly interested in subdomains in $C_\d$. Take $\Om_j\in C_\d$. Define the sub-region of $S_\d$ :
\beqo
S_{(r,r+\f{\d}{2})}=\{z\in S_\d:\, r<x<r+\f{\d}{2}\}, \quad r\in[-\frac{\d}{2},0].
\eeqo
Noted that for each $r$, $S_{(r,r+\f{\d}{2})}$ is a region with half width of $S_\d$. Obviously, there exists $r_j\in [-\frac{\d}{2},0]$ such that
\beq\label{halfu2}
\int_{\Om_j\cap S_{(r_j,r_j+\f{\d}{2})}}|u_j^2|\leq \f{1}{2}\int_{\Om_j} u_j^2.
\eeq
Let $\xi_j$ be a smooth cut-off function such that
\beqo
\xi_j(z)\equiv 1\text{ if }z\not\in S_{(r_j,r_j+\f{\d}{2})};\quad \xi_j(z)\equiv 0\text{ on }x=r_j+\f{\d}{4};\quad |\na \xi_j|\leq \f{8}{\d}.
\eeqo

\textbf{Claim:} If $\f{\int_{\Om_j}|\na(\xi_ju_j)|^2}{\int_{\Om_j}|\xi_j u_j|^2}\geq (1+\f{\e}{5})\lam_1(\Om_j)$, then there exists a constant $C_1$ which depends on $\e$, such that $\lam_1(\Om_j)\leq C_1(\e)$.

\emph{Proof of the Claim}: we calculate directly:
\begin{align}
\label{Ddomain}&\f{\int_{\Om_j}|\na(\xi_ju_j)|^2}{\int_{\Om_j}|\xi_j u_j|^2}\geq (1+\f{\e}{5})\lam_1(\Om_j)\\
\label{star}\Rightarrow & \int_{\Om_j} \left[ |\na\xi_j|^2u_j^2+2\xi_ju_j\na\xi_j\cdot\na u_j+\xi_j^2|\na u_j|^2 \right]\geq (1+\f{\e}{5})\lam_1(\Om_j)\int_{\Om_j}(u_j\xi_j)^2
\end{align}
By integration by parts, we have
\beqo
\int_{\Om_j}|\na u_j|^2\xi_j^2=\int_{\Om_j}\lam_1(\Om_j) (u_j\xi_j)^2-\int_{\Om_j}2\xi_j u_j \na\xi\cdot\na u_j.
\eeqo
Thus \eqref{star} implies
\beqo
\int_{\Om_j} |\na\xi_j|^2u_j^2\geq \f{\e}{5}\lam_1(\Om_j)\int_{\Om_j}(u_j\xi_j)^2
\eeqo
By assumption on $\xi_j$, we conclude that
\beqo
\lam_1(\Om_j)\leq\f{640}{\e \d^2}=:C_1(\e).
\eeqo

Let $D_\d$ be the subset of $C_\d$ that is consisted of all these sub-domains that satisfies \eqref{Ddomain}. According to the above claim, for any
$\Om_j\in
D_\d$, $\lam_1(\Om_j)\leq C_1(\e)$, and by the well-known Faber-Krahn inequality, there exists a constant $C_2(\e)$ such that $|\Om_j|\geq C_2(\e)$. Then
we can control the number of sub-domains in $D_\d$ by a constant only depends on $\e$, but is independent of $m$, i.e. $\#D_\d\leq C_3(\e)$.

Based on $u,\,A_\d,\,B_\d,\,C_\d,\,D_\d$, we can then define modified vector-valued functions $v,w$ such that $\mathrm{supp}(v)\subset D_1'\cup S_\d$
and $\mathrm{supp}(w)\subset D_2'\cup S_\d$. We follow the following schemes:
\begin{enumerate}
  \item[(i)] If $\Om_j\in A_\d$, then $v_j=u_j$;
  \item[(ii)] If $\Om_j\in B_\d$, then $w_j=u_j$;
  \item[(iii)] If $\Om_j\in C_\d\backslash D_\d$, then we have by definition
  \beq\label{ddreverse}
  \f{\int_{\Om_j}|\na(\xi_ju_j)|^2}{\int_{\Om_j}|\xi_j u_j|^2}\leq (1+\f{\e}{5})\lam(\Om_j)
  \eeq
  Note that $\xi=0$ on $\{x=r_j+\f{\d}{4}\}$, the line $\{x=r_j+\f{\d}{4}\}$ divides $\Om_j$ into two sub-domains $\Om_j^1, \Om_j^2$, where
  \beqo
  \Om_j^1\subset D_1'\cup S_\d,\quad \Om_j^2\subset D_2'\cup S_\d.
  \eeqo
  Moreover, we have $u_j\xi_j\big|_{\Om_j^1}\in H_0^1(\Om_j^1)$ and $u_j\xi_j\big|_{\Om_j^2}\in H_0^1(\Om_j^2)$. We denote
  \beqo
  \tau_1:=\f{\int_{\Om_j^1}|\na(\xi_ju_j)|^2}{\int_{\Om_j^1}|\xi_j u_j|^2},\quad \tau_2:=\f{\int_{\Om_j^2}|\na(\xi_ju_j)|^2}{\int_{\Om_j^2}|\xi_j
u_j|^2}.
  \eeqo
  Clearly \eqref{ddreverse} implies that
  \beqo
  \min\{\tau_1,\,\tau_2\}\leq (1+\f{\d}{5})\lam(\Om_j).
  \eeqo
  If $\tau_1\leq \tau_2$, then we let
  \beqo
  v_j=\f{\xi_ju_j}{\sqrt{\int_{\Om_j^1}|\xi_ju_j|^2}}\text{ on }\Om_j^1, \quad v_j=0 \,\text{ elsewhere}.
  \eeqo
  Otherwise, let
  \beqo
  w_j=\f{\xi_ju_j}{\sqrt{\int_{\Om_j^2}|\xi_ju_j|^2}}\text{ on }\Om_j^2, \quad w_j=0 \,\text{ elsewhere}.
  \eeqo
  We also denote
  \begin{align*}
  E_\d&=\{\Om_j^1: \, \Om_j\in C_\d\backslash D_\d,\, \tau_1\leq \tau_2\}\\
  F_\d&=\{\Om_j^2: \, \Om_j\in C_\d\backslash D_\d,\, \tau_1> \tau_2\}.
  \end{align*}
  \item[(iv)]
  Finally, we rearrange the vector of functions $v,\,w$ such that
  \begin{align*}
  &\mathrm{supp}\,v_j\neq \emptyset, \quad \mathrm{supp}\, v_j\in A_\d\cup E_\d, \text{ for all }j=1,...,m_1.\\
  &\mathrm{supp}\,w_j\neq \emptyset, \quad \mathrm{supp}\, w_j\in B_\d\cup F_\d, \text{ for all }j=1,...,m_2.
  \end{align*}
  Here $m_1=\#A_\d+\#E_\d$, $m_2=\#B_\d+\#F_\d$.
\end{enumerate}

Now we are ready to prove \eqref{finalineq}. One calculates
\begin{align}
\nonumber \f{\sum\limits_{j=1}^m \int_{\Omega_j}|\na u_j|^2}{m^2}&\geq \f{\sum\limits_{\Om_j\in A_\d}\int_{\Om_j}|\na u_j|^2+\sum\limits_{\Om_j\in B_\d}\int_{\Om_j}|\na u_j|^2+\sum\limits_{\Om_j\in C_\d\backslash D_\d}\int_{\Om_j}|\na u_j|^2}{m^2}\\
\nonumber &\geq \f{1}{m^2}\bigg( \sum\limits_{\Om_j\in A_\d}\int_{\Om_j}|\na u_j|^2+\sum\limits_{\Om_j\in B_\d}\int_{\Om_j}|\na u_j|^2+\f{\sum\limits_{\Om_j\in E_\d}\int_{\Om_j}|\na v_j|^2}{1+\e/5}+\f{\sum\limits_{\Om_j\in F_\d}\int_{\Om_j}|\na w_j|^2}{1+\e/5} \bigg)\\
\label{mid}&\geq \f{1}{m^2(1+\e/5)} \bigg( \sum\limits_{\Om_j\in A_\d}\int|\na v_j|^2+\sum\limits_{\Om_j\in E_\d}\int|\na v_j|^2+\sum\limits_{\Om_j\in B_\d}\int|\na w_j|^2+\sum\limits_{\Om_j\in F_\d}\int|\na w_j|^2 \bigg)
\end{align}
Define
\beqo
\tilde{D}_1=\bigcup\limits_{\Om_j\in A_\d\cup E_\d} \Om_j,\quad \tilde{D}_2=\bigcup\limits_{\Om_j\in B_\d\cup F_\d} \Om_j.
\eeqo
By the construction above we have
\beqo
\tilde{D}_i\subset D_i'\cup S_\d \subset  (1+\f{\e}{10})D_i, \quad \text{for }i=1,2
\eeqo
where $\d<\d(\e)$ is small enough.
Hence we obtain that
\begin{align}
\label{tildeD1}&\lim\limits_{m\ri\infty} l_m^1(\tilde{D}_1)\geq (1-\f{\e}{5})\lim\limits_{m\ri\infty} l_m^1(D_1)=\f{1-\e/5}{\al}c\\
\label{tildeD2} &\lim\limits_{m\ri\infty} l_m^1(\tilde{D}_2)\geq (1-\f{\e}{5})\lim\limits_{m\ri\infty} l_m^1(D_2)=\f{1-\e/5}{1-\al}c
\end{align}
Note that by our construction, $v\in H_0^1(\tilde{D}_1,\Sigma^{m_1})$ and $w\in H_0^1(\tilde{D}_2,\Sigma^{m_2})$, $m_1+m_2=m-C_3(\e)$. We take $m$ sufficiently large such that
\beq\label{determinem}
\left(\f{m-C_3(\e)}{m}\right)^2\geq 1-\f{\e}{5}, \quad l_{m_1}^1(\tilde{D}_1)\geq \f{1-\e/4}{\al} c,\quad l_{m_2}^1(\tilde{D}_2)\geq \f{1-\e/4}{1-\al} c.
\eeq
Here we have assumed that $m_1,m_2$ also go to infinity when $m$ goes to infinity. If the latter is not true, then it is even easier to conclude
\eqref{finalineq}, and we shall omit the details to the readers. By combining \eqref{mid}, \eqref{tildeD1}, \eqref{tildeD2} and \eqref{determinem} we
can deduce that
\begin{align*}
 \f{1}{m^2}(\sum\limits_{j=1}^m \int_{\Om_j} |\na u_j|^2)&\geq \f{1}{m^2(1+\e/5)}\left(m_1^2l_{m_1}^1(\tilde{D}_1)+m_2^2l_{m_2}^1(\tilde{D}_2)\right)\\
 &\geq \f{c(1-\e/4)}{(1+\e/5)m^2}\left( \f{m_1^2}{\al}+\f{m_2^2}{1-\al} \right)\\
 &\geq \f{1-\e/4}{1+\e/5}\left(\f{m-C_3(\e)}{m}\right)^2c\\
 &\geq \f{(1-\e/4)(1-\e/5)}{1+\e/5}c\geq (1-\e)c.
\end{align*}
It completes the proof.
\end{proof}

\end{document}